\def\timestamp{%
Time-stamp: <lift-to-M-final.tex: woensdag 07-05-2025 at 14:21:48 (cest)>}
\def\stripname Time-stamp: <#1: #2 #3 at #4 #5>{#3/#4 (#1)}
\edef\filedate{\expandafter\stripname\timestamp}
\documentclass{amsart}

\DeclareMathSymbol\B \mathord{AMSb}{`B}
\DeclareMathSymbol\HH \mathord{AMSb}{`H}
\DeclareMathSymbol\I \mathord{AMSb}{`I}
\DeclareMathSymbol\M \mathord{AMSb}{`M}
\DeclareMathSymbol\N \mathord{AMSb}{`N}

\newcommand\Hstar{\HH^*}
\newcommand\Mstar{\M^*}
\newcommand\Nstar{\N^*}
\newcommand\calB{\mathcal{B}}

\newcommand\calP{\mathcal{P}} \let\pow=\calP
\newcommand\calQ{\mathcal{Q}}

\newcommand\calBN{\calB^\N}
\newcommand\calBNfin{\calBN/\fin}
\newcommand\cee{\mathfrak{c}}
\newcommand\cl{\operatorname{cl}}
\newcommand\dom{\operatorname{dom}}
\newcommand\ran{\operatorname{ran}}
\newcommand\clbeta[1][X]{\cl_{\beta #1}}
\newcommand\orpr[2]{\langle{#1},{#2}\rangle}
\newcommand\preim{^\gets}
\DeclareMathSymbol\restr\mathbin{AMSa}{"16}
\DeclareMathSymbol\le    \mathrel{AMSa}{"36}
\DeclareMathSymbol\ge    \mathrel{AMSa}{"3E}

\newcommand\axiom{\mathsf}
\newcommand\CH{\axiom{CH}}
\newcommand\MA{\axiom{MA}}
\newcommand\OCA{\axiom{OCA}}
\newcommand\OCAT{\OCA_T}
\newcommand\PFA{\axiom{PFA}}
\newcommand\ZFC{\axiom{ZFC}}

\newcommand\fin{\mathit{fin}}

\mathchardef\mathhyphen "002D   
\newcommand\ulim[1][u]{#1\mathhyphen\!\lim}

\newcommand\flip{\mathsf{flip}}

\newtheorem{theorem}{Theorem}
\newtheorem{lemma}[theorem]{Lemma}
\newtheorem{proposition}[theorem]{Proposition}
\newtheorem{corollary}[theorem]{Corollary}

\theoremstyle{remark}
\newtheorem{remark}[theorem]{Remark}
\usepackage{amsrefs,amssymb,verbatim}

\newcommand\Zbl[1]{\relax
                   \ifhmode \unskip \spacefactor 3000 \space \fi 
                   Zbl~\MRhref {#1}{#1}}

\begin{document}

\title{$\Mstar$, $\Nstar$, and $\Hstar$}

\author{Will Brian}
\address{
W. R. Brian\\
Department of Mathematics and Statistics\\
University of North Carolina at Charlotte\\
Charlotte, NC, USA}
\email{wbrian.math@gmail.com}
\urladdr{wrbrian.wordpress.com}

\author{Alan Dow}
\address{
A. S. Dow\\
Department of Mathematics and Statistics\\
University of North Carolina at Charlotte\\
Charlotte, NC, USA}
\email{adow@charlotte.edu}
\urladdr{https://webpages.uncc.edu/adow}

\author{Klaas Pieter Hart}
\address{K. P. Hart\\
Facult EEMCS\\
TU Delft\\
Delft, the Netherlands}
\email{k.p.hart@tudelft.nl}
\urladdr{https://fa.ewi.tudelft.nl/\~{}hart}


\thanks{The first author is supported in part by NSF grant DMS-2154229.}

\begin{abstract}
Let $\M = \mathbb N \times [0,1]$. The natural projection $\pi: \M \rightarrow \mathbb N$, which sends $(n,x)$ to $n$, induces a projection mapping $\pi^*: \M^* \rightarrow \mathbb N^*$, where $\M^*$ and $\mathbb N^*$ denote the \v{C}ech-Stone remainders of $\M$ and $\mathbb N$, respectively. 

We show that $\mathsf{CH}$ implies every autohomeomorphism of $\mathbb N^*$ lifts through the natural projection to an autohomeomorphism of $\M^*$. That is, for every homeomorphism $h: \mathbb N^* \rightarrow \mathbb N^*$ there is a homeomorphism $H: \M^* \rightarrow \M^*$ such that $\pi^* \circ H = h \circ \pi^*$. This complements a recent result of the second author, who showed that this lifting property is not a consequence of $\mathsf{ZFC}$. 

Combining this lifting theorem with a recent result of the first author, we also prove that $\mathsf{CH}$ implies there is an order-reversing autohomeomorphism of~$\mathbb H^*$, the \v{C}ech-Stone remainder of the half line $\mathbb H = [0,\infty)$. 
\end{abstract}
\dedicatory{To the memory of Peter Nyikos}
\date\filedate
\maketitle

\section{Introduction}

Let $\I$ denote the unit interval $[0,1]$, and let $\M=\N\times\I$. 
Let $\pi:\M\to\N$ denote the natural projection $(n,x) \mapsto n$. 
Moving to the \v{C}ech-Stone compactifications, $\pi$~extends to a continuous surjection $\beta\pi: \beta\M \to \beta\N$. 
Because $\pi$ is surjective and $\pi\preim(n)$~is compact for every $n$, 
the map~$\beta\pi$ restricts to a surjection 
from $\Mstar = \beta\M \setminus \M$ onto $\Nstar = \beta\N \setminus \N$. 
Let $\pi^* = \beta\pi \restr \Mstar$. 
In other words, $\pi^*$ is the continuous surjection $\Mstar \to \Nstar$ 
induced by~$\pi$. 

This paper is organized around two main theorems. 
The first states that, assuming the Continuum Hypothesis (henceforth $\CH$), 
every autohomeomorphism of~$\Nstar$ can be ``lifted'' through~$\pi^*$ to an 
autohomeomorphism of~$\Mstar$: 

\begin{theorem}\label{thm:main}
Assuming $\CH$, if $h$ is an autohomeomorphism 
of $\Nstar$, then there is an autohomeomorphism $H$ of $\Mstar$ 
such that $\pi^* \circ H = h \circ \pi^*$.
\end{theorem} 

This theorem complements a recent result of the second author in~\cite{Alan}, 
where he shows that the conclusion of Theorem~\ref{thm:main} fails 
consistently. 
Specifically, this lifting property fails in the model of Veli\v{c}kovi\'c 
from~\cite{Boban} in which $\MA_{\aleph_1}$ holds and $\Nstar$~has a 
nontrivial autohomeomorphism. 
Theorem~\ref{thm:main} and the main theorem of~\cite{Alan}, taken together, 
answer Question~2.4 in~\cite{A&KP}.

It is not difficult to see that all trivial autohomeomorphisms of $\Nstar$ lift through $\pi^*$. (We give the easy argument in Section~\ref{sec:OCA}.) 
Thus the conclusion of Theorem~\ref{thm:main} holds if all autohomeomorphisms 
of~$\Nstar$ are trivial, which is implied by forcing axioms like~$\PFA$, or even just~ $\OCAT$. 
Thus this lifting property provides a rare example of a statement about \v{C}ech-Stone remainders that follows from forcing axioms and from~$\CH$, but not from $\ZFC$. 

Our second theorem concerns yet another \v{C}ech-Stone remainder. Let $\HH = [0,\infty)$ denote the space of nonnegative real numbers, and let $\Hstar = \beta\HH \setminus \HH$.  
The order on $\HH$ induces a quasiorder on certain subsets of $\Hstar$. 
This is explained further in Section~\ref{sec:H} below (see also the survey of $\Hstar$ by the third author, \cite{KP}). 
An old folklore question about $\Hstar$ is whether there is an autohomeomorphism of $\Hstar$ that reverses this order. Our second theorem shows that, consistently, there is.

\begin{theorem}\label{thm:meroeht}
$\CH$ implies there is an order-reversing autohomeomorphism of $\Hstar$. 
\end{theorem}

On the other hand, a recent result of Vignati 
(see \cite{Vignati}*{Theorem~C}) states that $\OCAT+\MA$ implies all autohomeomorphisms of $\Hstar$ are trivial. An easy argument (which is given in Section~\ref{sec:H} below) shows that all trivial autohomeomorphisms of $\Hstar$ are order-preserving. Thus the existence of an order-reversing autohomeomorphism of $\Hstar$ is independent of~$\ZFC$. 

Theorem~\ref{thm:meroeht} could really be called a corollary. It follows relatively easily from two other theorems: Theorem~\ref{thm:main} stated above, and a recent result of the first author in \cite{Brian}, which states that, assuming $\CH$, the shift map and its inverse are conjugate in the autohomeomorphism group of $\Nstar$. 

The next two sections are devoted to $\Mstar$ and $\Nstar$, and the proof of Theorem~\ref{thm:main}. The fourth and final section of the paper contains some background material on $\Hstar$ and a proof of Theorem~\ref{thm:meroeht}.

\section{More on $\Mstar$ and $\Nstar$}\label{sec:OCA}

The aim of this section is to introduce some ideas and notation concerning $\Mstar$ and $\Nstar$, and to prove two relatively easy positive results similar to Theorem~\ref{thm:main}. These two results form part of the motivation for proving Theorem~\ref{thm:main}. 

An \emph{almost permutation} of $\N$ is a bijection from one co-finite subset of $\N$ to another. A \emph{trivial autohomeomorphism} of $\Nstar$ is a homeomorphism $h$ induced by an almost permutation $f$ of $\N$, in the sense that $h = \beta f \restr \Nstar$, or equivalently, the action of $h$ on the clopen subsets of $\Nstar$ is simply 
$h[A^*] = (f[A])^*$, for all $A \subseteq \N$. 
Similarly, if $f$ is a homeomorphism between two co-compact subsets of $\M$ then $H_f = \beta f \restr \Mstar$ is an autohomeomorphism of $\Mstar$, and any such autohomeomorphism of $\Mstar$ is called \emph{trivial}. 

\begin{proposition}\label{prop:triviality}
If $h$ is a trivial autohomeomorphism 
of $\Nstar$, then there is a trivial autohomeomorphism $H$ of $\Mstar$ 
such that $\pi^* \circ H = h \circ \pi^*$.
\end{proposition}
\begin{proof}
Let $h$ be a trivial autohomeomorphism of $\Nstar$. Fix an almost permutation $f$ of $\N$ such that $h = \beta f \restr \Nstar$. Define $g: \mathrm{dom}(f) \times \I \to \M$ by setting 
$g(n,x) \,=\, (f(n),x).$ Observe that $g$ is a homeomorphism from one co-compact subset of $\M$ to another, and therefore $H_g = \beta g \restr \Mstar$ is a trivial autohomeomorphism of $\Mstar$. 
Because $\pi \circ g = f \circ \pi$, we have $\beta \pi \circ \beta g = \beta f \circ \beta \pi$, and then restricting to the \v{C}ech-Stone remainders, $\pi^* \circ H_g = h \circ \pi^*$.
\end{proof} 

The existence of a nontrivial autohomeomorphism of $\Nstar$ is independent of $\ZFC$. On the one hand, Walter Rudin proved in \cite{Rudin} that $\CH$ implies there are $2^{\cee}$ autohomeomorphisms in total, though of course only $\cee$ of them can be trivial. On the other hand, Shelah proved in 
\cite{Shelah}*{Chapter 4}, via an oracle-c.c.\ iteration, that it is consistent to have all autohomeomorphisms trivial. Building on Shelah's work, Shelah and Stepr\={a}ns showed in \cite{Shelah&Steprans} that $\PFA$ implies all autohomeomorphisms are trivial. Veli\v{c}kovi\'c showed in \cite{Boban} that $\OCAT+\MA$ suffices, though it is consistent with $\MA_{\aleph_1}$ to have nontrivial autohomeomorphisms. (Here $\OCAT$ denotes Todor\v{c}evi\'c's Open Coloring Axiom, defined in \cite{Todorcevic}, now sometimes called $\axiom{OGA}$.) Building on work of Moore in \cite{Justin}, DeBondt, Farah, and Vignati showed in \cite{DFV} that $\OCAT$ alone implies all autohomeomorphisms of $\Nstar$ are trivial.
Combined with Proposition~\ref{prop:triviality}, this shows Theorem~\ref{thm:main} remains true when $\CH$ is replaced with $\OCAT$. 

\begin{corollary}
Assuming $\OCAT$, if $h$ is an autohomeomorphism 
of $\Nstar$, then there is an autohomeomorphism $H$ of $\Mstar$ 
such that $\pi^* \circ H = h \circ \pi^*$. \qed
\end{corollary}

\noindent In light of this, the main point of Theorem~\ref{thm:main} is not simply that the conclusion is consistent, but specifically that it follows from $\CH$. This is good to know for two reasons: because it contributes to the longstanding program of understanding the behavior of \v{C}ech-Stone remainders under $\CH$, and because it enables us to prove Theorem~\ref{thm:meroeht}. Note that the conclusion of Theorem~\ref{thm:meroeht} is not implied by $\OCAT$ (see Proposition~\ref{prop:Htriv} below), and indeed, we do not currently know how to obtain an order-reversing autohomeomorphism of $\Hstar$ except via $\CH$.

\smallskip

For each $n \in \N$, let $\I_n=\pi\preim(n) = \{n\}\times\I$. These are the connected components of $\M$. 
Analogously, for each $u\in\Nstar$ let 
$\I_u = (\pi^*)\preim(u)$. 
Equivalently, 
$$
\I_u \,=\, \bigcap_{A \in u} (\pi^*)\preim\bigl[\clbeta[\N]A\bigr] 
     \,=\, \bigcap_{A \in u} \clbeta[\M] ( \pi\preim[A]).
$$
These are the connected components of $\Mstar$ (see \cite{KP}*{Corollary~2.2}). 

In particular, if $H$ is an 
autohomeomorphism of $\Mstar$, then $H$ permutes the set $\{\I_u:u\in\Nstar\}$ of connected components of $\Mstar$. 
Let $\rho_H$ denote the corresponding permutation of $\Nstar$, so that $H(\I_u) = \I_{\rho_H(u)}$ for all $u \in \Nstar$.

If $B^*$ is clopen in $\Nstar$, then $(\pi^*)\preim[B^*]$ is clopen in $\Mstar$, by the continuity of $\pi^*$. But the converse is also true: if $C$ is clopen in $\Mstar$, then $C = (\pi^*)\preim[B^*]$ for some clopen $B^* \subseteq \Nstar$. To see this, note that if $C$ is clopen then $C$ and $\Mstar \setminus C$ are both compact, so $\pi^*[C]$ and $\pi^*[\Mstar \setminus C]$ are both compact as well. But these sets are disjoint, because $(\pi^*)\preim(u) = \I_u$ is connected for each $u \in \Nstar$, which means either $\I_u \subseteq C$ or $\I_u \subseteq \Mstar \setminus C$. 
Hence $\pi^*[C]$ and $\pi^*[\Mstar \setminus C]$ are complementary closed sets, hence clopen, and $C = (\pi^*)\preim\left[\pi^*[C]\right]$ and $\Mstar \setminus C = (\pi^*)\preim\left[\pi^*[\Mstar \setminus C]\right]$. 

\begin{proposition}
If $H$ is an autohomeomorphism of $\Mstar$, then $\rho_H$ is an autohomeomorphism of $\Nstar$
such that $\pi^*\circ H = \rho_H\circ\pi^*$. 
\end{proposition}
\begin{proof}
That $\pi^*\circ H = \rho_H\circ\pi^*$ follows from the definition of $\rho_H$, so we need only show $\rho_H$ is an autohomeomorphism of $\Nstar$. Because $\rho_H$ is bijective (again, by definition) and $\Nstar$ is compact, it suffices to show $\rho_H$ is continuous. 
Let $A \subseteq \N$, so that $A^*$ is a basic clopen subset of $\Nstar$. Then $(\pi^*)\preim[A^*]$ is a clopen subset of $\Mstar$. Because $H$ is a homeomorphism, $H\preim\bigl[(\pi^*)\preim[A^*]\bigr]$ is clopen as well. By the paragraph preceding this proposition, this means there is some clopen $B^* \subseteq \Nstar$ such that $H\preim\bigl[(\pi^*)\preim[A^*]\bigr] = (\pi^*)\preim[B^*]$. But 
$$
H\preim\bigl[(\pi^*)\preim[A^*]\bigr] = (\pi^* \circ H)\preim [A^*] = 
(\rho_H \circ \pi^*)\preim [A^*] = (\pi^*)\preim\bigl[\rho_H \preim[A^*]\bigr],
$$
so $(\pi^*)\preim \bigl[\rho_H \preim[A^*]\bigr] = (\pi^*)\preim[B^*]$, which implies $\rho_H \preim[A^*] = B^*$.
\end{proof}

In other words, this proposition states that autohomeomorphisms of $\Mstar$ project downward through $\pi^*$ to autohomeomorphisms of $\Nstar$. 
That is, Theorem~\ref{thm:main} remains true, without even needing to assume $\CH$, if we switch the roles of $\Mstar$ and $\Nstar$. 
The opposite direction, lifting upward through $\pi^*$ rather than projecting downward, is more subtle, and this more difficult direction is the content of Theorem~\ref{thm:main}.

\section{A proof of Theorem~\ref{thm:main}}

\noindent\emph{Proof of Theorem~\ref{thm:main}:} 
Let $h$ be an autohomeomorphism of~$\Nstar$. 
Using $\CH$, we aim to construct an autohomeomorphism $H$ of $\Mstar$ such 
that $\pi^* \circ H = h \circ \pi^*$. 
Our construction of $H$ needs a few ingredients.

\smallskip
The first is a map $h^+:\pow(\N)\to\pow(\N)$ with the property that for all 
subsets $A$ of~$\N$ we have $h[A^*]=h^+(A)^*$.

\smallskip
The second is a suitable base for the closed sets of~$\Mstar$ that is a
distributive lattice with respect to $\cup$ and~$\cap$.
We shall describe~$H$ dually by specifying an automorphism of that base. 

\smallskip
Let $\calB$ be a countable distributive lattice base for the closed sets 
of~$\I$, say the lattice generated by the family of closed intervals with 
rational end points.
We identify members of $\calBN$ with closed subsets of~$\M$ in the obvious way:
if $B\in\calBN$ then
$$
F_B=\bigcup\bigl\{\{k\}\times B(k):k\in\N\bigr\}
$$ 
In this way the power $\calBN$ determines a base for the closed sets
of~$\M$, and hence also for the closed sets of~$\beta\M$, as the following
lemma implies.

\smallskip
For the ring-theoretic approach one could take the subring~$R$ of the 
ring~$C(\I)$ generated by the constant functions and the functions~$d_B$ 
for $B\in\calB$, where $d_B(x)=d(x,B)$.
The power~$R^\N$ represents a subring of~$C(\M)$ and the bounded elements
of it would be the analogue of~$\calBN$ in what follows.

\begin{lemma}
If $F$ and $G$ are closed and disjoint subsets of~$\M$ then there
are members $B$ and~$C$ of~$\calBN$ such that
$F\subseteq F_B$, $G\subseteq F_C$ and $F_B\cap F_C=\emptyset$.  \qed
\end{lemma}

Because for closed subsets $F$ and $G$ of~$\M$ we 
have $\clbeta[\M]F\cap\Mstar=\clbeta[\M]G\cap\Mstar$ if and only if 
$\{n:F\cap\I_n\neq G\cap\I_n\}$ is finite, we see that the reduced power
$\calBNfin$ determines a base for the closed sets of~$\Mstar$:
the family $\{F_B^*:B\in\calBN\}$.
We have $F_B^*\subseteq F_C^*$ if and only if $\{k:B(k)\subseteq C(k)\}$ is cofinite.
The latter condition also defines the partial order of~$\calBNfin$.

It follows that if we let $B^*$ denote the equivalence class of~$B\in\calBN$
in~$\calBNfin$ we get the equivalence
$$
B^*\le C^* \text{ \quad if and only if \quad} F_B^*\subseteq F_C^*
$$
for $B,C\in\calBN$.

The algebraic structure of the lattice $\calBNfin$ is determined completely
by its partial order, so it will suffice to define an automorphism
of the partially ordered set $(\calBNfin,\le)$.

We define this automorphism by defining a partial map 
$\varphi:\calBN\to\calBN$ with the following properties. (The construction is detailed below.)
\begin{enumerate}
\item If $B\in\calBN$ then there are unique $C\in\dom\varphi$ 
      and $D\in\ran\varphi$
      such that $B=^*C$ and $B=^*D$; this uniqueness ensures that $\varphi$ 
      determines a well-defined surjection from $\calBNfin$ to itself.
\item If $B$ and $C$ are in $\dom\varphi$, and if  
      $A_1=\{k:B(k)\subseteq C(k)\}$ and 
      $A_2=\{k:C(k)\subseteq B(k)\}$, 
      then $\{k:\varphi(B)(k)\subseteq\varphi(C)(k)\}=^*h^+(A_1)$
      and \linebreak $\{k:\varphi(C)(k)\subseteq\varphi(B)(k)\}=^*h^+(A_2)$.
\end{enumerate}

In condition~(2) the set $A_3=\N\setminus(A_1\cup A_2)$ is the set of~$k$
where $B(k)$ and $C(k)$ are incomparable.
It follows that $h^+(A_3)$ is mod finite equal to the set of~$k$
where $\varphi(B)(k)$ and $\varphi(C)(k)$ are incomparable.

\begin{lemma}
The conditions above ensure that $\varphi$ induces an automorphism
of the partially ordered set $(\calBNfin,\le)$.  
\end{lemma}

\begin{proof}
Since $\dom\varphi$ and $\ran\varphi$ intersect every equivalence class in 
exactly one point we automatically obtain a surjective map~$\varphi^*$ 
from~$\calBNfin$ to itself.

In order to see that $\varphi^*$ is injective assume $\varphi^*(B^*)=\varphi^*(C^*)$.
This means that $\{k:\varphi(B)(k)=\varphi(C)(k)\}$ is co-finite,
hence so are $\{k:\varphi(B)(k)\subseteq\varphi(C)(k)\}$ and 
$\{k:\varphi(B)(k)\subseteq\varphi(C)(k)\}$.
But this means, with the notation as in condition~(2) above,
that $h^+(A_1)$ and $h^+(A_2)$ are co-finite too. 
Because $h^+$ represents~$h$ the sets $A_1$ and $A_2$ must then be co-finite
as well, and we conclude that $\{k:B(k)=C(k)\}$ is co-finite and hence 
that $B=C$ even.

Similarly, using that $A_1$~is cofinite if and only if $h^+(A_1)$ is co-finite, we obtain
that $B^*\le C^*$ iff $\varphi(B^*)\le \varphi^*(C^*)$,
and so $\varphi^*$ is an automorphism.
\end{proof}

Given $A\subseteq\N$, define $\Mstar_A = \left(\bigcup_{n\in A}\I_n\right)^* = \clbeta[\M]\left(\bigcup_{n\in A}\I_n\right) \setminus \M$.

\begin{lemma}\label{lemma.3}
The autohomeomorphism~$H$ of\/~$\Mstar$ determined by~$\varphi^*$ 
satisfies $h\circ\pi^*=\pi^*\circ H$.
\end{lemma}

\begin{proof}
To show that $\pi^*\circ H=h\circ\pi^*$ we let $A\subseteq\N$ and show that
$$
\pi^*\bigl[H[\Mstar_A]\bigr] = h\bigl[\pi^*[\Mstar_A]\bigr].
$$
Define $B$ and $I$ in $\calBN$ by
$$
B(k)=\begin{cases} 
         \I        &\text{ if }k\in A\\ 
         \emptyset &\text{ if }k\notin A
    \end{cases}
$$
and $I(k)=\I$ for all~$k$; 
in our construction we shall have $I\in\dom\varphi$ and $\varphi(I)=I$.

Then $A=\{k: I(k)\subseteq B(k)\}$ and so 
$$
\{k:\I\subseteq \varphi(B)(k)\}
=\{k:\varphi(I)(k)\subseteq\varphi(B)(k)\}=^*h^+(A).
$$
Likewise $\N\setminus A=\{k:B(k)\subseteq\emptyset\}$ and we obtain 
$$
\{k:\varphi(B)(k)=\emptyset\}=^*h^+(\N\setminus A)=^*\N\setminus h^+(A).
$$
We deduce that $H[\Mstar_A]=\Mstar_{h^+(A)}$ and so
$$
\pi^*\bigl[H[\Mstar_A]\bigr]=\pi^*[\Mstar_{h^+(A)}]=h^+(A)^*=h[A^*]=
h\bigl[\pi^*[\Mstar_A]\bigr].   \qedhere
$$
\end{proof}

\subsection*{The construction}

Using $\CH$, fix an enumeration $\langle B_\alpha:\alpha\in\omega_1\rangle$
of~$\calBN$. 
In a recursion of length $\omega_1$ we construct two sequences
$\langle C_\alpha:\alpha\in\omega_1\rangle$ and 
$\langle D_\alpha:\alpha\in\omega_1\rangle$
of members of~$\calBN$. 
These will be such that 
$$
\varphi=\bigl\{\orpr{C_\alpha}{D_\alpha}:\alpha\in\omega_1\bigr\}
$$
is the map that we seek.

To begin the construction we let 
\begin{itemize}
\item $C_0=D_0=\langle\emptyset:n\in\N\rangle$, and 
\item $C_1=D_1=\langle\I:n\in\N\rangle$. 
\item $C_2=D_2=\langle[0,\frac{1}{2}]:n\in\N\rangle$. 
\end{itemize}
The first two conditions take care of the maximum and minimum of $\calBNfin$ and 
facilitate the proof of Lemma~\ref{lemma.3}. 
The third condition ensures that our autohomeomorphism $H$ will preserve the order on each of the $\I_u$. We have not yet said what this order is, and that is because one does not need to worry about it yet. 
This third condition does not affect the rest of the proof in this section, but we include it because it will be useful in Section~\ref{sec:H} when finding an order-reversing autohomeomorphism of $\Hstar$.

Next let $\alpha\ge2$ and assume we have 
$\varphi_\alpha=\bigl\{\orpr{C_\beta}{D_\beta}:\beta\in\alpha\bigr\}$ such that
$\varphi_\alpha$ satisfies conditions~(1) and~(2) above up to~$\alpha$,
that is
\begin{enumerate}
\item If $\gamma<\beta<\alpha$ then 
      $\{k:C_\gamma(k)\neq C_\beta(k)\}$ and $\{k:D_\gamma(k)\neq D_\beta(k)\}$
      are infinite; this ensures the uniqueness clause in~(1).
\item If $\gamma<\beta<\alpha$ and  
      $A_1=\{k:C_\gamma(k)\subseteq C_\beta(k)\}$ and 
      $A_2=\{k:C_\gamma(k)\subseteq C_\beta(k)\}$
      then $\{k:D_\gamma(k)\subseteq D_\beta(k)\}=^*h^+(A_1)$
      and $\{k:D_\gamma(k)\subseteq D_\beta(k)\}=^*h^+(A_2)$.
\end{enumerate}

We extend $\varphi_\alpha$ to $\varphi_{\alpha+1}$, as follows.

\begin{itemize}
\item If $\alpha$~is even let $C_\alpha$ be the first term of the
      sequence $\langle B_\alpha:\alpha\in\omega_1\rangle$ that satisfies
      $B_\alpha\neq^* C_\beta$ for all $\beta<\alpha$.
      We show how to determine $D_\alpha$ so as to satisfy the conditions above
      up to and including~$\alpha$.
\item If $\alpha$ is odd let $D_\alpha$ be the first term of the
      sequence $\langle B_\alpha:\alpha\in\omega_1\rangle$ that satisfies
      $B_\alpha\neq^* D_\beta$ for all $\beta<\alpha$.
      We show how to determine $C_\alpha$ so as to satisfy the conditions above
      up to and including~$\alpha$.
\end{itemize}

We shall only deal with the even case; the argument in the odd case is 
the mirror image of that in the even case.

\smallbreak
We stop before we start, however.
It turns out that the second assumption on the recursion is too weak.

To illustrate this assume that there are $\gamma$ and $\beta$ below~$\alpha$
such that the set~$A$ of $k$ such that 
$C_\gamma(k)\subseteq C_\alpha(k)\subseteq C_\beta(k)$
is infinite.
Then we shall need that 
$D_\gamma(k)\subseteq D_\alpha(k)\subseteq D_\beta(k)$
for all but finitely many~$k\in h^+(A)$.
We shall certainly have $D_\gamma(k)\subseteq D_\beta(k)$ for all but finitely
many~$k\in h^+(A)$, so there seems to be no problem specifying~$D_\alpha(k)$
for these values of~$k$.

But it is very well possible that for all $k\in A$ 
(or at least infinitely many)
the set $C_\alpha(k)$ is a subset of the \emph{interior} of~$C_\beta(k)$.
In that case there is a~$\delta$ such that for all these~$k$ we have
$B_\delta(k)\cup C_\beta(k)=\I$ and $B_\delta(k)\cap C_\alpha(k)=\emptyset$.

If the first such $\delta$ is (much) larger than~$\alpha$ then it seems
likely that we only ensured $C_\gamma(k)\subseteq C_\beta(k)$ for 
enough $k\in h^+(A)$, but possibly not that $C_\gamma(k)$~is a subset of 
the \emph{interior} of~$C_\beta(k)$.

Because ``$C_\alpha(k)$~is in the interior of~$C_\beta(k)$'' is expressible
in terms of the lattice operations and hence in terms of the order
we should have ``$D_\alpha(k)$~is in the interior of~$D_\beta(k)$''
often enough as well.
But the latter is impossible in case we did not ensure enough times
that $D_\gamma(k)$ is in the interior of~$D_\beta(k)$.

\smallbreak
To fix this problem we need to find a way to ``look ahead'' to later stages of the construction, but without explicitly using ordinals 
larger than~$\alpha$. 
That way is via quantifiers and elementarity.

To stay with our example we note that ``$B(k)$~is in the interior of~$C(k)$''
is expressible as $(\exists x)\psi\bigl(B(k),C(k),x\bigr)$, 
where $\psi(y,z,x)$ is ``$y\cap x=\emptyset\land z\cup x=\I$''.

So, given $\gamma,\beta<\alpha$ we should also look at
$A=\{k:(\exists x)\psi(C_\gamma(k),C_\beta(k),x(k))\}$ and ensure
that also $h^+(A)=^*\{k:(\exists x)\psi(D_\gamma(k),D_\beta(k),x(k))\}$.
This then will help us build $D_\alpha$ such that $D_\alpha(k)$~is in the
interior of~$D_\beta(k)$ often enough.

\smallskip
We strengthen condition~(2) so that it covers all formulas of lattice theory
and all finite sets of ordinals. 
\begin{itemize}
\item[$(*)_\alpha$] For every formula $\chi$ in the language of lattice theory
      with free variables~$x_1$, \dots, $x_n$, 
      and for every tuple $(\beta_1,\ldots,\beta_n)$ of ordinals 
      below~$\alpha$, the set 
      $A=\bigl\{k:\chi\bigl(C_{\beta_1}(k),\ldots,C_{\beta_n}(k)\bigr)\bigr\}$
      satisfies
    $h^+(A)=^*\bigl\{k:\chi\bigl(D_{\beta_1}(k),\ldots,D_{\beta_n}(k)\bigr)\bigr\}$.
\end{itemize}
In the example above we would have $\chi(x_1,x_2)$ equal to
$(\exists x)\psi(x_1,x_2,x)$.

The sets $\emptyset$ and $\I$ are the interpretations of the constants~$0$
and~$1$ of the language of lattice theory in~$\calB$ and we have that same
lattice~$\calB$ in every coordinate.
It follows that for every formula under consideration that involve
only~$\emptyset$ and\slash or~$\I$ the set of~$k$s where the formula holds 
is either empty or equal to~$\omega$.
This means that $(*)_2$ holds and we have a solid basis for our recursion.

\subsection*{An application of elementary equivalence and saturation}

Before we start to build $D_\alpha$ we need an intermediate result.
This will involve some model theory, especially elementary equivalence
and saturation.
The results that we need can be found in~\cite{MR1221741}*{Chapter~10} 
or~\cite{MR1462612}*{Chapter~8}. 

We fix $u\in\Nstar$ for the moment, put $v=h(u)$, and consider the ultrapowers 
$\B_u=\calBN/u$ and $\B_v=\calBN/v$.

\smallbreak
The structures $(\B_u,\bar C)$ and $(\B_v,\bar D)$ are elementarily equivalent.
Here $\bar C$ is the sequence of elements of~$\B_u$ determined
by the sequence $\langle C_\beta:\beta\in\alpha\rangle$,
and likewise $\bar D$ is determined in~$\B_v$
by $\langle D_\beta:\beta\in\alpha\rangle$.

The reason is that when $\chi$ is a formula with free variables
$x_1$, \dots, $x_n$ and if $\beta_1$, \dots, $\beta_n$ are members of~$\alpha$
such that $\B_u\models\chi(C_{\beta_1},\ldots,C_{\beta_n})$
then the set $A_1=\{k:\calB\models \chi(C_{\beta_1}(k),\ldots,C_{\beta_n}(k))\}$ 
belongs to~$u$.
Then $h^+(A_1)$ belongs to~$v$, and hence so does
$A_2=\{k: \calB\models \chi(D_{\beta_1}(k),\ldots,D_{\beta_n}(k)) \}$.
But this then implies that $\B_v\models\chi(D_{\beta_1},\ldots,D_{\beta_n})$.

\smallbreak
The ultrapower~$\B_v$~is saturated and so \cite{MR1221741}*{Lemma~10.1.3} 
or~\cite{MR1462612}*{Lemma~8.1.3} applies, 
which guarantees the existence of an element~$D$ of~$\B_v$ such that 
$(B_u,\bar C, C_\alpha)$ and $(\B_v,\bar D, D)$ are elementarily equivalent.

This means that a local version of~$(*)_{\alpha+1}$ holds at the points~$u$ 
and~$v$, with $D$ in place of~$D_\alpha$:
\begin{quote}
  if $\chi$ is a formula from the language of lattices with free variables
$x_1$, \dots $x_n$, $x_{n+1}$, and if $\beta_1$, \dots, $\beta_n$ are members 
of~$\alpha$ 
then the two sets 
$A_1=\{k: \calB\models \chi(C_{\beta_1}(k),\ldots,C_{\beta_n}(k), C_{\alpha}(k))\}$ 
and  
$A_2=\{k: \calB\models \chi(D_{\beta_1}(k),\ldots,D_{\beta_n}(k), D(k))\}$
satisfy $A_1\in u$ if and only if $A_2\in v$. 
\end{quote}

For every~$u\in\Nstar$ we choose $D_u$ such that
$(B_u,\bar C, C_\alpha)$ and $(\B_v,\bar D, D_u)$ are elementarily 
equivalent.

Now let $\chi$ be a formula, with free variables $x_1$, \dots, $x_n$,
and let $\beta_1$, \dots, $\beta_n$ be elements of~$\alpha+1$.
By the rules of interpretation we know that for every $u\in\Nstar$
we have \emph{either}  $\B_u\models\chi(C_{\beta_1},\ldots,C_{\beta_n})$,
\emph{or} $\B_u\models\neg\chi(C_{\beta_1},\ldots,C_{\beta_n})$.

Then the former implies that 
$\B_v\models\chi(D_{\beta_1},\ldots,D_{\beta_n})$,
and the latter implies 
that $\B_v\models\neg\chi(D_{\beta_1},\ldots,D_{\beta_n})$, where,
in both cases, we insert $D_u$ at the positions where $\beta_i=\alpha$.

Translated to subsets of $\N$ this becomes
if $A=\{k:\calB\models\chi(C_{\beta_1}(k),\ldots,C_{\beta_n}(k))\}$ belongs to~$u$
then $\{k:\calB\models\chi(D_{\beta_1}(k),\ldots,D_{\beta_n}(k))\}$ and $h^+(A)$
both belong to~$v$.

\subsection*{Making $D_\alpha$}

We build $D_\alpha$ out of bits and pieces of the elements $D_u$ chosen above.

The idea will be to make~$D_\alpha$ in a recursion of length~$\omega$
each time adding finitely many coordinates of finitely many~$D_u$ in such
a way that the higher the coordinates the more formulas these decide.
In the end $D_\alpha$ should then decide every formula almost everywhere.

We also have to take care of the ordinals below~$\alpha$ so we start
by fixing an enumeration of the set of pairs $\orpr\chi{\bar\beta}$
of formulas and finite sequences of ordinals in~$\alpha+1$ as
$\bigl<\orpr{\chi_m}{\bar\beta_m}:m\in\N\bigr>$. 
We assume that the number of free variables in~$\gamma_m$ is always equal to 
the length of the sequence~$\bar\beta_m$, call this number~$p_m$.

Each pair $\orpr{\chi_m}{\bar\beta_m}$ determines a partition of $\N$
into two sets
$$
A_{m,0}=\{k:\calB\models\neg\chi_m(C_{\beta_{m,1}}(k),\ldots,C_{\beta_{m,p_m}}(k))\}
$$ and
$$
A_{m,1}=\{k:\calB\models\chi_m(C_{\beta_{m,1}}(k),\ldots,C_{\beta_{m,p_m}}(k))\}.
$$
Using the enumeration we make a sequence of 
partitions $\langle\calP_m:m\in\N\rangle$ of~$\N$, as follows.

For each sequence $s\in2^m$ let $A_s=\bigcap_{l< m}A_{l,s(l)}$, 
and put $\calP_m=\{A_s:s\in2^m\}$. 
Thus $\calP_0=\{\N\}$.

The map~$h^+$ transforms these partitions in almost-partitions, that is,
the union $\bigcup\calP_m$ is co-finite, and if $s,t\in2^m$ and $s\neq t$ 
then $h^+(A_s)\cap h^+(A_t)$ is finite.
This implies that there is a natural number $N_m$ such that
$\{h^+(A)\setminus N_m:A\in\calP_m\}$ is a partition of~$\N\setminus N_m$.

In fact, by raising $N_m$ if nesessary we can ensure that the map
$A\mapsto h^+(A)\setminus N_m$ is an isomorphism between the Boolean
algebras generated by $\{A_{l,i}:l<m$, and $i\in2\}$ and
$\{h^+(A_{l,i})\setminus N_m:l<m$, and $i\in2\}$ respectively.

\subsubsection*{Other partitions}

Let us fix $m$ for the time being.

The definition of $h^+$ implies that for every $u\in\Nstar$ and~$s\in2^m$
we have $A_s\in u$ iff $h^+(A_s)\in h(u)$.

By the local version of~$(*)_{\alpha+1}$ above we know that for every~$u$ and~$s$
we have $A_s\in u$ iff the set $B_{u,s}$ belongs to~$h(u)$, where
$B_{u,s}$ is the set of those~$k$ that satisfy for all $l<m$:
\begin{itemize}
\item $\calB\models\chi_l(D_{\beta_{l,1}}(k),\ldots,D_{\beta_{l,p_l}}(k))$ 
      when $s(l)=1$, and 
\item $\calB\models\neg\chi_l(D_{\beta_{l,1}}(k),\ldots,D_{\beta_{l,p_l}}(k))$ 
      when $s(l)=0$.
\end{itemize}
In both cases we substitute $D_u$ for $D_{\beta_{l,i}}$ 
whenever $\beta_{l,i}=\alpha$.

Note that this implies that $h^+(A_s)\cap B_{u,s}\in h(u)$ iff $A_s\in u$.
It follows that if $s\in2^m$ then the family $\{B_{u,s}^*:u\in A_s^*\}$
covers $h[A_s^*]$, hence there is a finite subset~$F_s$ of~$A_s^*$ such
that $\{B_{u,s}^*:s\in F_u\}$ covers~$h[A_s^*]$.
This means that $h^+(A_s)\setminus\bigcup_{u\in F_s}B_{u,s}$ is finite. 

We can shrink the sets $B_{u,s}$ so that 
\begin{itemize}
\item $B_{u,s}\subseteq h^+(A_s)$ for all~$u\in F_s$
\item the sets $h^+(A_s)\setminus\bigcup_{u\in F_s}B_{u,s}$ remain finite
\end{itemize}
as a consequence the whole family
$\calQ_m=\bigcup_{s\in2^m}\{B_{u,s}\setminus N_m:u\in F_s\}$ is pairwise disjoint,
and its union is co-finite; we increase $N_m$ if necessary so that
the latter union contains~$\N\setminus N_m$.
In short, the family~$\calQ_m$ is a partition of $\N\setminus N_m$ that
is a refinement of~$\{h^+(A):A\in\calP_m\}$.

\subsubsection*{Building $D_\alpha$}

Using the construction above we obtain a sequence 
$\langle\calQ_m:m\in\N\rangle$ of almost-partitions and
a sequence $\langle N_m:m\in\N\rangle$ of natural numbers such that
$\{Q\setminus N_m:Q\in\calQ_m\}$ is a partition of $N\setminus N_m$ and,
without loss of generality $N_m<N_{m+1}$ for all~$m$.

As indicated above we build $D_\alpha$ piece by piece, more precisely,
for every~$m$ we define $D_\alpha$ on the interval~$[N_m,N_{m+1})$
using the $D_u$ with $u\in\bigcup\{F_s:s\in2^m\}$.

Fix such an~$m$.
If $k\in[N_m,N_{m+1})$ then there is one pair~$(s,u)$ with $s\in2^m$ and
$u\in F_s$ such that $k\in B_{u,s}$.
Define $D_\alpha(k)=D_u(k)$.
Then, by the very choice of~$B_{u,s}$ we get, for all $l<m$:
\begin{itemize}
\item $\calB\models\chi_l(D_{\beta_{l,1}}(k),\ldots,D_{\beta_{l,p_l}}(k))$ 
      when $s(l)=1$, and 
\item $\calB\models\neg\chi_l(D_{\beta_{l,1}}(k),\ldots,D_{\beta_{l,p_l}}(k))$ 
      when $s(l)=0$.
\end{itemize}
By the choice of~$N_m$ above we find that for $\orpr li\in m\times2$ we have
$h^+(A_{l,i})\setminus N_m=
 \bigcup\{B_{u,s}\setminus N_m:s\in2^m, u\in F_s, s(l)=i\}$.
It follows that
$$
h^+(A_{l,1})\cap[N_m,N_{m+1})=
\bigl\{k:\calB\models\chi_l(D_{\beta_{l,1}}(k),\ldots,D_{\beta_{l,p_l}}(k))\bigr\}
     \cap[N_m,N_{m+1})
$$
and
$$
h^+(A_{l,0})\cap[N_m,N_{m+1})=
\bigl\{k:\calB\models\neg\chi_l(D_{\beta_{l,1}}(k),\ldots,D_{\beta_{l,p_l}}(k))\bigr\}
     \cap[N_m,N_{m+1})
$$

\subsubsection*{Verification of $(*)_{\alpha+1}$}

Let $l\in\N$; we show that $(*)_{\alpha+1}$ holds for the pair 
$\orpr{\chi_l}{\bar\beta_l}$.
We have 
$$
A_{l,1}=\{k:\calB\models\chi_l(C_{\beta_{l,1}}(k),\ldots,C_{\beta_{l,p_l}}(k))\},
$$
$$
B_{l,1}=\{k:\calB\models\chi_l(D_{\beta_{l,1}}(k),\ldots,D_{\beta_{l,p_l}}(k))\}.
$$
We must show that $B_{l,1}=^*h^+(A_{l,1})$.

But our construction above ensures that 
$B_{l,1}\cap[N_m,N_{m+1})=h^+(A_{l,1})$ whenever $m>l$.
This clearly suffices, and this completes the proof. 
\qed

\begin{remark}
The proof above is based on Wallman's representation theorem for
distributive lattices, see~\cite{MR1503392}.
That paper established a many-valued duality between certain distributive
lattices and compact spaces: to every distributive lattice there corresponds 
a compact space, and to every compact space there may correspond various 
lattices, e.g., the full family of closed sets, or any base for its closed
sets that forms a lattice.

A true duality for compact spaces is due to Gelfand and Kolomogorov
in~\cite{zbMATH03035106}: here the ring of continuous functions is the
algebraic counterpart of the compact space.
One can give a version of our proof based on this duality, where the
ring~$C(\Mstar)$ of continuous functions on~$\Mstar$ is represented
as the quotient of the ring~$C^*(\M)$ by the ideal of functions that
have limit~zero at infinity.
One can give a version of our proof that constructs an automorphism
of~$C^(\Mstar)$ whose dual is an autohomeomorphism as desired.
\end{remark}

\section{An order-reversing autohomeomorphism of $\Hstar$}\label{sec:H}

We begin this section with a description of the standard quasiorder on the connected components $\I_u$ of $\Mstar$. 

Given a sequence $\bar x = \langle x_n :\, n \in \N \rangle \in \I^\N$ and $u \in \Nstar$, define
$$
\bar x_u = \ulim\nolimits_n  (n,x_n),
$$
where the limit is taken in $\beta\M$. Equivalently, this is the unique element of the set $\bigcap_{A \in u} \clbeta[\M] \{(n,x_n) :\, n \in A\}$. 
Let $C_u \subseteq \I_u$ denote the set of all points of this form:
$$C_u = \{ \bar x_u :\, \bar x \in \I^\N\}.$$
This is a proper subset of $\I_u$ (see \cite{KP}). 
Observe that $C_u$ has a natural linear order: 
$$\bar x_u \leq_u \bar y_u \quad \Leftrightarrow \quad \{ n :\, x_n \leq y_n \} \in u.$$
The fact that this is a total order of $C_u$ can be seen as an instance of Ło\'s' Theorem, because $C_u$ is really just the ultrapower $\I^\N / u$. 
Observe that $C_u$ has a least element $\bar 0_u$ and a greatest element $\bar 1_u$, where $\bar 0$ denotes the constant sequence $\langle 0,0,0,\dots \rangle$ and $\bar 1$ the constant sequence $\langle 1,1,1,\dots \rangle$. 

\begin{proposition}\label{prop:order}
The set $C_u \setminus \{\bar 0_u,\bar 1_u\}$ is dense in $\I_u$, and its subspace topology is the same as the order topology induced by $\leq_u$. 
Furthermore, 
\begin{enumerate}
\item $\I_u$ is irreducible between $\bar 0_u$ and $\bar 1_u$, and if $p \in C_u \setminus \{\bar 0_u,\bar 1_u\}$ then $\I_u \setminus \{p\}$ has two connected components, one containing $\bar 0_u$ and the other $\bar 1_u$. 
\item If $x,y \in C_u$ and $x <_u y$, then every subcontinuum of $\I_u$ that contains $y$ and $\bar 0_u$ also contains $x$, and inversely, every subcontinuum of $\I_u$ that contains $x$ and $\bar 1_u$ also contains $y$.
\end{enumerate}
\end{proposition}

\noindent See \cite{KP}*{Section~2} for a proof. 
This proposition implies that $\{\bar 0_u,\bar 1_u\}$ is a topologically definable subset of $\I_u$. In particular, any autohomeomorphism $H: \Mstar \to \Mstar$ must map $\{\bar 0_u,\bar 1_u\}$ to $\{\bar 0_{\rho_H(u)},\bar 1_{\rho_H(u)}\}$.

Part $(2)$ of this proposition enables us to extend the total order $\leq_u$ on $C_u$ to a quasiorder on $\I_u$, also denoted $\leq_u$: for any $x,y \in \I_u$, define 
$x \leq_u y$ if and only if every subcontinuum of $\I_u$ that contains $y$ and $\bar 0_u$ also contains $x$, if and only if every subcontinuum of $\I_u$ that contains $x$ and $\bar 1_u$ also contains $y$.

Let us note that $\leq_u$ is not a partial order. 
Let us write $x \equiv_u y$ to mean that $x \leq_u y$ and $y \leq_u x$. 
For each $x \in \I_u$, the set $L_x = \{y \in \I_u :\, y \equiv_u x\}$ is 
called the \emph{layer} of $x$ in $\I_u$. 
If $x \in C_u$, then it turns out $L_x = \{x\}$. 
For points $x\in\I_u\setminus C_u$ there are two possibilities:
$L_x=\{x\}$ or $|L_x| = 2^{\cee}$.
There are always~$x$ for which the second possibility obtains;
under~$\CH$ there are points outside~$C_u$ that have a one-point layer,
but in the Laver model all points outside~$C_u$ have non-trivial layers,
see~\cite{MR1216810}.
The layers of $\I_u$ can be quite large, and topologically complex: all are 
indecomposable continua for example. 
The quotient of $\I_u$ by $\equiv_u$ is the Dedekind completion of $(C_u,\leq_u)$, with its usual order topology. So one may think of $\I_u$ as something like the Dedekind completion of $C_u$, but where some of the gaps have been filled not with single points, but with complicated compacta.

An  autohomeomorphism $H: \Mstar \to \Mstar$ is \emph{order-preserving} if 
$$\text{if $x,y \in \I_u$ and $x \leq_u y$, then $H(x) \leq_{\rho_H(u)} H(y)$},$$
and $H$ is \emph{order-reversing} if 
$$\text{if $x,y \in \I_u$ and $x \leq_u y$, then $H(y) \leq_{\rho_H(u)} H(x)$}.$$ 
Equivalently, $H$ is order-preserving if $H(\bar 0_u) = \bar 0_{\rho_H(u)}$ and $H(\bar 1_u) = \bar 1_{\rho_H(u)}$ for all $u \in \Nstar$, and it is order-reversing if $H(\bar 0_u) = \bar 1_{\rho_H(u)}$ and $H(\bar 1_u) = \bar 0_{\rho_H(u)}$ for all $u$.

Observe that an autohomeomorphism $H$ of $\Mstar$ need be neither order-preserving nor order-reversing: it may have $H(\bar 0_u) = \bar 0_{\rho_H(u)}$ for some $u$ and $H(\bar 0_u) = \bar 1_{\rho_H(u)}$ for other $u$. For example, given $A \subseteq \N$, consider the homeomorphism $\flip_A: \M \to \M$ that flips the intervals in $A \times \I$:
$$\flip_A(n,x) \,=\, \begin{cases}
(n,1-x) &\text{ if $n \in A$}, \\
(n,x) &\text{ if $n \notin A$}.
\end{cases}$$
 The trivial autohomeomorphism of $\Mstar$ induced by $\flip_A$ is order-preserving if $A$ is finite, it is order-reversing if $A$ is co-finite, and it is neither order-preserving nor order-reversing if $A$ and $\N \setminus A$ are both infinite. 

\begin{lemma}\label{lem:main+}
Assuming $\CH$, if $h$ is an autohomeomorphism of $\Nstar$, then there is an order-preserving autohomeomorphism $H$ of $\Mstar$ such that $\pi^* \circ H = h \circ \pi^*$. \hfill $\square$
\end{lemma}

\noindent This lemma merely re-states Theorem~\ref{thm:main}, but with the added requirement that $H$ be order-preserving. To prove the lemma, simply note that our proof of Theorem~\ref{thm:main} in the previous section already produces an order-preserving map. (The construction ensures that $H$ maps $K = \left( \N \times [0,\frac{1}{2}] \right)^*$ to itself, which means that for every $u$, $\bar 0_u \in K$ cannot map to $\bar 1_{\rho_H(u)} \notin K$, and therefore must map to $\bar 0_{\rho_H(u)}$.)

\smallskip

Observe that $\HH$ is obtained naturally as a topological quotient of $\M$: just glue the rightmost point of each connected component $\I_n$ to the leftmost point of the next component $\I_{n+1}$. More precisely, define an equivalence relation $\sim$ on $\M$ by setting $(n,1) \sim (n+1,0)$ for all $n \in \N$ (and of course $(n,x) \sim (n,x)$ for all $(n,x) \in \M$). The quotient space $\M / \!\sim$ is $\HH$. 

One can obtain $\Hstar$ as a quotient of $\Mstar$ in a similar fashion. 
First, let $\sigma$ denote the autohomeomorphism of $\Nstar$ induced by the successor map $n \mapsto n+1$ (which is an almost permutation of $\N$). Explicitly, if $u \in \Nstar$ then $\sigma(u)$ is the ultrafilter generated by $\{ A+1 :\, A \in u \}$. Next, like with $\M$ and $\HH$, our quotient mapping $\Mstar \to \Hstar$ is defined by gluing the rightmost point of each connected component $\I_u$ to the leftmost point of the ``next" component, $\I_{\sigma(u)}$. More precisely, define an equivalence relation $\sim$ on $\Mstar$ by setting $\bar 1_u \sim \bar 0_{\sigma(u)}$ for all $u \in \Nstar$ (and of course $x \sim x$ for all $x \in \Mstar$). Then the quotient space $\Mstar / \!\sim$ is $\Hstar$ 
(see \cite{KP}*{Theorem~2.4}). 
Let $Q: \Mstar \to \Hstar$ denote this quotient mapping. 


For each $u \in \N^*$, our quotient mapping $Q: \Mstar \to \Hstar$ restricts to an injection on $\I_u$. 
Because $\I_u$ is compact and $Q$ continuous, this means $Q \!\restriction\! \I_u$ is an embedding $\I_u \to \Hstar$. 
In other words, we may (and do) think of the $\I_u$ as subspaces of $\Hstar$. 
For each $u \in \N^*$, let $\I^Q_u = Q[\I_u]$ denote this copy of $\I_u$ in $\Hstar$. 
Alternatively, 
$$
\I^Q_u = \bigcap_{A \in u} \clbeta[\HH] \bigcup\{[n,n+1] :\, n \in A\}.
$$
Via this identification of $\I_u$ with $\I^Q_u$, each $\I^Q_u \subseteq \Hstar$ has a natural quasi-order, the push-forward of the quasi-order $\leq_u$ on $\I_u$, which we still denote by $\leq_u$ in $\I^Q_u$. 

Let us note that the $\I^Q_u$ are examples of what are called \emph{standard subcontinua} of $\Hstar$. These are special connected subsets of $\Hstar$ whose structure and interrelationships determine much about the topology of $\Hstar$ (see \cite{KP}{Section~5}). It is not difficult to see that the ordering $\leq_u$ described above for $\I^Q_u$ matches the usual quasi-order defined on a standard subcontinuum of $\Hstar$.

Let us say that a homeomorphism $H: \Hstar \to \Hstar$ is \emph{order-reversing} if there is a permutation $\rho$ of $\N^*$ such that 
\begin{itemize}
\item[$\circ$] $H[\I^Q_u] = \I^Q_{\rho(u)}$ for all $u \in \N^*$, and 
\item[$\circ$] if $x \leq_u y$ in $\I^Q_u$, then $H(y) \leq_{\rho(u)} H(x)$ in $\I^Q_{\rho(u)}$.
\end{itemize}
In other words, an order-reversing autohomeomorphism of $H$ is one that permutes the $\I^Q_u$ while reversing their order. 

\begin{proposition}\label{prop:Htriv}
No trivial autohomeomorphism of $\Hstar$ is order-re\-vers\-ing. 
\end{proposition}
\begin{proof}
Let $H$ be a trivial autohomeomorphism of $\Hstar$. We aim to show $H$ is not order-reversing.  Because $H$ is trivial, there is a homeomorphism $f: C \to D$, where $C$ and $D$ are co-compact subsets of $\HH$, such that $H = \beta f \!\restriction\! \Hstar$. 
Observe that $f$ must be order-preserving on a tail of $\HH$: i.e., if $a < b$ and $a,b$ are sufficiently large, then $f(a) < f(b)$. 

For each $n \in \N$, fix $a_n,b_n$ such that $n \leq a_n < b_n \leq n+1$. 
Let $\bar a = \langle a_n :\, n \in \N \rangle$ and $\bar b = \langle b_n :\, n \in \N \rangle$, and let 
$f(\bar a) = \langle f(a_n) :\, n \in \N \rangle$ and $f(\bar b) = \langle f(b_n) :\, n \in \N \rangle$. 

Fix $u \in \N^*$ and suppose $H[\I^Q_u] = \I^Q_v$ for some $v \in \N^*$. (Otherwise $H$ is not order-reversing.) 
Let $x = Q(\bar a_u)$ and $y = Q(\bar b_u)$. Because $a_n < b_n$ for all $n$, we have $\bar a_u <_u \bar b_u$ in $\I_u$, which means $x = Q(\bar a_u) <_u Q(\bar b_u) = y$ in $\I^Q_u$. 

Now $H(x) = \beta f(Q(\bar a_u)) = Q(f(\bar a)_v)$, and similarly $H(y) = Q(f(\bar b)_v)$. 
But because $f$ is order-preserving on a tail, we have $f(a_n) < f(b_n)$ for all sufficiently large $n$, and therefore $f(\bar a)_v \leq_v f(\bar b)_v$. Hence $H(x) = Q(f(\bar a)_v) \leq_v Q(f(\bar b)_v) = H(y)$ in $\I^Q_v$, and this means that $H$ is not order-reversing.
\end{proof}

\begin{corollary}
It is consistent that no autohomeomorphism of $\Hstar$ is order-re\-vers\-ing.
\end{corollary}
\begin{proof}
As we mentioned already in Section 2, a recent result of Vignati in \cite{Vignati} states that $\mathsf{OCA}_T+\mathsf{MA}$ implies all autohomeomorphisms of $\Hstar$ are trivial. The corollary follows from this and the previous proposition. 
\end{proof}

\begin{proof}[Proof of Theorem~\ref{thm:meroeht}]
As before, let $\sigma$ denote the trivial autohomeomorphism of $\Nstar$ induced by the successor function $n \mapsto n+1$ on $\N$. 
By a recent theorem of the first author (the main theorem of \cite{Brian}), $\CH$ implies $\sigma$ and $\sigma^{-1}$ are conjugate in the autohomeomorphism group of $\Nstar$. In other words, $\CH$ implies there is an autohomeomorphism $f$ of $\Nstar$ such that $f \circ \sigma = \sigma^{-1} \circ f$. 
Fix some such $f$ and, using $\CH$ again and applying Lemma~\ref{lem:main+}, fix an order-preserving autohomeomorphism $F$ of $\Mstar$ such that $\pi^* \circ F = f \circ \pi^*$. 

Recall the homeomorphism $\flip_\N: \M \to \M$ is defined by  
$$\flip_\N(n,x) \,=\, (n,1-x)$$
for all $(n,x) \in \M$. Let $\flip_\N^* = \beta\, \flip_\N \restr \Mstar$ denote the trivial autohomeomorphism of $\Mstar$ induced by $\flip_\N$. 
Clearly $\flip_\N^*$ is order-reversing on each $\I_u$, and in particular,  
$$\flip_\N^*(\bar 1_u) = \bar 0_u \qquad \text{ and } \qquad \flip_\N^*(\bar 0_u) = \bar 1_u$$
for every $u \in \Nstar$.

Because $\flip_\N^*$ is an order-reversing autohomeomorphism of $\Mstar$ and $F$ is an order-preserving autohomeomorphism of $\Mstar$, their composition
$$H = F \circ \flip_\N^*$$
is an order-reversing autohomeomorphism of $\Mstar$. 
In particular, observe that 
$$
H(\bar 1_u) = F \circ \flip_\N^*(\bar 1_u) = F(\bar 0_u) = \bar 0_{f(u)}
$$
and
$$
H(\bar 0_u) = F \circ \flip_\N^*(\bar 0_u) = F(\bar 1_u) = \bar 1_{f(u)}
$$
for every $u \in \Nstar$. 

Recall the equivalence relation $\sim$ on $\Mstar$ described earlier in this section, whose non-singleton equivalence classes are the sets of the form $\{\bar 1_u,\bar 0_{\sigma(u)}\}$ for $u \in \Nstar$. 
By our choice of $f$, for any given $u \in \Nstar$ we have 
$$H(\bar 1_u) = \bar 0_{f(u)} \qquad \text{ and } \qquad H(\bar 0_{\sigma(u)}) = \bar 1_{f(\sigma(u))} = \bar 1_{\sigma^{-1}(f(u))}.$$
In particular, $H$ maps the $\sim$-equivalence class $\{\bar 1_u,\bar 0_{\sigma(u)}\}$ to the $\sim$-equivalence class $\{\bar 1_{\sigma^{-1}(f(u))},\bar 0_{f(u)}\}$ (i.e., the class $\{\bar 1_v,\bar 0_{\sigma(v)}\}$ where $v = \sigma^{-1}(f(u))$). In other words, $H$ respects the equivalence classes of the relation $\sim$. Consequently, 
$$[x]_\sim \mapsto [H(x)]_\sim$$ 
is a well-defined mapping $\Mstar/\!\sim\ \to \Mstar/\!\sim$. Let $h$ denote this mapping. Recalling that $Q$ denotes the quotient mapping from $\Mstar$ to $\Mstar / \!\sim$, we have $Q \circ H = h \circ Q$.

We claim $h$ is an order-reversing autohomeomorphism of $\Hstar \cong \Mstar/\!\sim$. 
The fact that $h$ is an autohomeomorphism follows from the fact that $H$ is, and that $h \circ Q = Q \circ H$. 
Next, recall that for every $u \in \Nstar$, the quotient mapping $Q$ restricts to a homeomorphism $\I_u \to \I^Q_u$. 
Because the sets of the form $\I_u$ are the connected components of $\Mstar$, the homeomorphism $H$ maps each $\I_u$ homeomorphically to $\I_{\rho_H(u)} = \I_{f(u)}$. 
Thus, for each $u \in \Nstar$, the map $h = Q \circ H \circ Q^{-1}$ is a composition of homeomorphisms $\I^Q_u \to \I_u \to \I_{f(u)} \to \I^Q_{f(u)}$. Thus, setting $\rho = f$, $h$ satisfies the first clause in the definition of an order-reversing autohomeomorphism of $\Hstar$.

Finally, we wish to show that 
if $x,y \in \I^Q_u$ and $x \leq_u y$ then $h(y) \leq_{f(u)} h(x)$. 
Fix $u \in \Nstar$ and $x,y \in \I^Q_u$ with $x \leq_u y$. 
As in the previous paragraph, 
$h = Q \circ H \circ Q^{-1}$ is a composition of homeomorphisms $\I^Q_u \to \I_u \to \I_{f(u)} \to \I^Q_{f(u)}$. Both $Q$ and $Q^{-1}$ preserve the order of $x$ and $y$, while $H$ reverses it, so $h(y) \leq_{f(u)} h(x)$.
\end{proof}

\begin{remark}
Peter had a knack of coming up with colourful descriptions of various 
constructions.
We thought it apt to give a Peter-esque description of the order-reversing 
homeomorphism in this fashion. 

One can think of $\M$ as a sequence of domino tiles. 
One obtains the map~$Q:\M\to\HH$, described above by tipping all tiles over
to the right so that for every~$n$ the top of~$\I_n$ touches the 
bottom of~$\I_{n+1}$; after some welding the map~$Q$ is done.

If one tips all tiles over to the left and again does some 
welding to join the bottom of $\I_n$ and the top of~$\I_{n+1}$ one obtains
another map from~$\M$ onto~$\HH$: the composition $Q\circ\flip$.

Both maps yields maps from $\Mstar$ onto~$\Hstar$ that may be interpreted
as tipping over the domino tiles~$\I_u$, either all to the right, or all to 
the left, and doing the analogous welding.

Our results show that under $\CH$ there is a autohomeomorphism of~$\Hstar$
that rearranges the components~$\I_u$ and flips them all over.

Clearly there is no autohomeomorphism of~$\HH$ itself that does this for the
tiles~$\I_n$, hence this autohomeomorphism is non-trivial.
\end{remark}

\end{document}